\documentclass[centertags,12pt]{amsart}
\usepackage{latexsym}
\usepackage{amsthm}
\usepackage{amssymb}
\usepackage{color}
\usepackage[utf8]{inputenc}
\usepackage{mathrsfs}
\usepackage[english]{babel}
\usepackage{setspace}
\usepackage{textcomp}
\usepackage{graphicx}
\usepackage{cite}
\usepackage{cancel}
\usepackage{comment}
\usepackage{hyperref}
\graphicspath{ {images/} }

\numberwithin{equation}{section}
\makeatletter
\renewcommand{\subsection}{\@startsection
{subsection}{2}{0mm}{\baselineskip}{-0.25cm}
{\normalfont\normalsize\bf}}
\makeatother

\newtheorem{theorem}{Theorem}[section]

\newtheorem{lemma}[theorem]{Lemma}
\newtheorem{corollary}[theorem]{Corollary}

\newtheorem*{theorem*}{Theorem}
\newtheorem{definition}[theorem]{Definition}   
\newtheorem{remark}[theorem]{Remark}

\begin{document}

\author[D. Bartoli]{Daniele Bartoli}
\address{ Dipartimento di Matematica e Informatica, Universit\`a degli Studi di Perugia, Perugia, Italy} \email{daniele.bartoli@unipg.it}

\author[M. Montanucci]{Maria Montanucci}
\address{ Department of Applied Mathematics and Computer Science, Technical University of Denmark, Kongens Lyngby, Denmark} \email{marimo@dtu.dk}

\author[Giovanni Zini]{Giovanni Zini}
\address{ Dipartimento di Matematica e Fisica, Universit\`a degli Studi della Campania Luigi Vanvitelli, Caserta, Italy} \email{giovanni.zini@unicampania.it}

\title[On certain self-orthogonal AG codes]{On certain self-orthogonal AG codes with applications to Quantum error-correcting codes}

\thanks{{\em 2010 Math. Subj. Class.}: 94B27, 11T71, 81P70, 14G50}


\begin{abstract}
In this paper a construction of quantum codes from self-orthogonal algebraic geometry codes is provided. Our method is based on the CSS construction as well as on some peculiar properties of the underlying algebraic curves, named Swiss curves. Several classes of well-known algebraic curves with many rational points turn out to be Swiss curves. Examples are given by Castle curves, GK curves, generalized GK curves and the Abd\'on-Bezerra-Quoos maximal curves. Applications of our method to these curves are provided. Our construction extends a previous one due to Hernando, McGuire, Monserrat, and Moyano-Fern\'andez.
\end{abstract}

\maketitle

\section{Introduction}

Since the discovery of quantum algorithms, such as a polynomial time algorithm for factorization by Shor \cite{Shor} and a quantum search algorithm by Grover \cite{Grover}, quantum computing has received a lot of attention. Even though a concrete and practical implementation of these algorithms is far away, it has nonetheless become clear that some form of error correction is required to protect quantum data from noise. This was the motivation for the development of quantum computation and, more specifically, of quantum error-correcting codes.

In the last decades much research has been done to find good quantum codes following several strategies and underlying mathematical structures. However, the most remarkable result is probably the one obtained by Calderbank and Shor \cite{CSH},  and Steane \cite{STE}; see also \cite{10}. Indeed they  showed  that quantum codes can be derived from classical linear error-correcting codes provided that certain orthogonality properties are satisfied, including Euclidean and Hermitian self-orthogonality; see \cite{10,28,NC}. This method, known as CSS construction,  has allowed to find many powerful quantum stabilizer codes.

Among all the classical codes used to produce quantum stabilizer codes, Algebraic-Geometry (AG) codes \cite{G1982} have received considerable attention \cite{MST, MTT, LGP, SSSSSOCODESSS, BMZ, MTZ,  MPL, CHe, GAHE, c1,c2,c3, c4,d1,d2}.
The interest towards AG codes is due to several reasons.
First, every linear code can be realized as an algebraic geometry code\cite{PE}. Also, AG codes were indeed used to improve the Gilbert-Varshamov bound \cite{GV}, an outstanding result at that time. Finally, conditions for Euclidean self-orthogonality of AG codes are well known \cite{BO} and allow us to translate the pure combinatorial nature of this problem into geometrical terms concerning the structure of the curves involved and their corresponding function fields.  

Castle curves and AG codes from them \cite{MST} give rise to good quantum error-correcting codes. Indeed, among all curves used to get AG codes, Castle and weak Castle curves combine the good properties of having a reasonable simple handling and giving codes with excellent parameters. This is confirmed by the fact that most of the best one-point AG codes studied in the literature belong to the family of Castle codes.

In \cite{MTT}, Munuera, Ten\'orio and Torres used the good properties of algebraic-geometry codes coming from Castle and weak Castle curves to provide new sequences of self-orthogonal codes. Their construction was extended in \cite{SSSSSOCODESSS} by Hernando, McGuire, Monserrat, and Moyano-Fern\'andez, who provided a way to obtain self-orthogonal AG codes, and hence good quantum codes, from a more general class of curves, strictly including Castle curves.
In this paper we further generalize the family of curves considered  in \cite{SSSSSOCODESSS} 
to what we call Swiss curves. The geometric properties on the underlying plane curves considered 
in \cite{SSSSSOCODESSS}  are weakened, focusing on the algebraic structure of the curves, that is, on their function field. The family of Swiss curves, and more generally of $r$-Swiss curves, includes the most studied and known families of algebraic curves with many rational points over finite fields. Some example are given by the Giulietti-Korchm\'aros curve \cite{GK2009}, the two generalized Giulietti-Korchm\'aros curves \cite{GGS2010} and \cite{BM}, as well as the Abd\'on-Bezerra-Quoos curve \cite{ABQ}.
Explicit constructions of quantum codes from these curves are provided, as well as comparisons with   the quantum Gilbert-Varshamov bound.
 
The paper is organized as follows.
Section \ref{sec:preliminaries} recalls basic notions on AG codes and quantum codes; in particular, we present some constructions from the literature where quantum codes are obtained from AG codes with self-orthogonality properties.
Section \ref{sec:swiss} defines a class of curves, namely Swiss curves, for which we prove in Theorem \ref{Main:Swiss} a result about self-orthogonality properties.
This is applied in Section \ref{sec:appl} to several curves which are shown to be Swiss and which provide quantum codes.
The results of Section \ref{sec:swiss} are generalized in Section \ref{sec:r-swiss} to a larger class of curves, called $r$-Swiss curves, and then applied in Section \ref{sec:r-appl} to generalized GK curves over finite fields of even order.
Finally, we note in Section \ref{sec:comp} that certain stabilizer quantum codes constructed in the previous sections are pure and exceed the quantum Gilbert-Varshamov bound.

\section{AG codes and quantum codes}\label{sec:preliminaries}

\subsection{AG codes}

We introduce here some basic notions on AG codes; for a detailed introduction to this topic, we refer to \cite[Chapter 2]{Sti}.

Let $\mathbb{F}_q$ be the finite field of order $q$ and $\mathcal{X}$ be a projective, absolutely irreducible, algebraic curve of genus $g$ defined over $\mathbb{F}_q$.
Let $\mathbb{F}_q(\mathcal{X})$ be the field of rational functions on $\mathcal{X}$ and $\mathcal{X}(\mathbb{F}_q)$ be the set of rational places of $\mathcal{X}$.
For any divisor $D=\sum_{P\in\mathcal{X}(\overline{\mathbb{F}}_q)}n_P P$ on $\mathcal{X}$, we denote by $v_P(D)$ the weight $n_P\in\mathbb{Z}$ of $P$ in $D$ (also called the valuation of $D$ at $P$), and by ${\rm supp}(D)$ the support of $D$, that is the finite set of places with non-zero weight in $D$; the degree of $D$ is $\deg(D)=\sum_{P\in{\rm supp}(D)} n_P$.
The Riemann-Roch space $\mathcal{L}(D)$ of an $\mathbb{F}_q$-rational divisor $D$ is the finite dimensional $\mathbb{F}_q$-vector space 
$$ \mathcal{L}(D)=\{f\in\mathbb{F}_q(\mathcal{X})\setminus\{0\}\colon (f)+D\geq0\}\cup\{0\},$$
 where $(f)=(f)_0-(f)_{\infty}$ denotes the principal divisor of $f$; here, $(f)_0$ and $(f)_\infty$ are respectively the zero divisor and the pole divisor of $f$.
The $\mathbb{F}_q$-dimension of $\mathcal{L}(D)$ is denoted by $\ell(D)$.

Let $\{P_1,\ldots,P_N\}\subseteq\mathcal{X}(\mathbb{F}_q)$ with $P_i\ne P_j$ for $i\ne j$, $D$ be the $\mathbb{F}_q$-rational divisor $P_1+\cdots+P_N$, and $G$ be an $\mathbb{F}_q$-rational divisor of $\mathcal{X}$ such that ${\rm supp}(D)\cap{\rm supp}(G)=\emptyset$.
Consider the $\mathbb{F}_q$-linear evaluation map 
\begin{eqnarray*}
e_D:&\mathcal{L}(G)&\to\mathbb{F}_q^N\\
&f&\mapsto(f(P_1),\ldots,f(P_N)).
\end{eqnarray*}

The (functional) AG code $C(D,G)$ is defined as the image $e_D(\mathcal{L}(G))$ of $e_D$.
The code $C(D,G)$ has parameters $[N,k,d]_q$ which satisfy $k=\ell(G)-\ell(G-D)$ and $d\geq N-\deg(G)$.
If $\deg(G)<N$, then $e_D$ is injective and $k=\ell(G)$.
If $2g-2<\deg(G)<N$, then $k=\deg(G)+1-g$.

The (Euclidean) dual code $C(D,G)^\bot$ has parameters $[N^\bot,k^\bot,d^\bot]_q$, where $N^\bot=N$, $k^\bot=N-k$, and $d^\bot\geq \deg(G)-2g+2$. Note that, if $2g-2<\deg(G)<N$, then $k^\bot=N-\deg(G)+g-1$.

\subsection{Quantum codes}

The main ingredient to construct quantum codes in this paper is the so-called {\it CSS construction} (named after Calderbank, Shor and Steane) which enables to construct quantum codes from classical linear codes; see \cite[Lemma 2.5]{LGP}.

A $q$-ary quantum code $Q$ of length $N$ and dimension $k$ is defined to be a $q^k$-dimensional Hilbert subspace of a $q^N$-dimensional Hilbert space $\mathbb H=(\mathbb C^q)^{\otimes n}=\mathbb C^q\otimes\cdots\otimes\mathbb C^q$. If $Q$ has minimum distance $D$, then $Q$ can correct up to $\lfloor\frac{D-1}{2}\rfloor$ quantum errors.
The notation $[[N,k,D]]_q$ is used to denote such a quantum code $Q$.
For an $[[N,k,D]]_q$-quantum code the quantum Singleton bound holds, that is, the minimum distance satisfies $D\leq 1+(N-k)/2$.
The quantum Singleton defect is $\delta^Q:=N-k-2D+2\geq0$, and the relative quantum Singleton defect is $\Delta^Q:=\delta^Q/N$.
If $\delta^Q=0$, then the code is said to be quantum MDS.
For a detailed introduction on quantum codes see \cite{LGP} and the references therein.

Another important bound for quantum codes is an analogue of the Gilbert-Varshamov bound. 
\begin{theorem}{\rm \cite[Theorem 1.4]{FM2004}}
Suppose that $N>k\geq2$, $d\geq 2$, and $N\equiv k \pmod 2$. Then there exists a pure stabilizer quantum code with parameters $[[N,k,d]]_q$ provided that 
\begin{equation}\label{Dis:GV}
\frac{q^{N-k+2}-1}{q^2-1}>\sum_{i=1}^{d-1}(q^2-1)^{i-1}\binom{N}{i}.
\end{equation}
\end{theorem}

\begin{lemma}{\rm \cite{10,28,NC}} \label{ccs} {\rm (CSS construction)}
Let $C_1$ and $C_2$ denote two linear codes with parameters $[N,k_i,d_i]_q$, $i=1,2$, and assume that $C_1 \subset C_2$. Then there exists an $[[N,k_2-k_1,D]]_q$ code with $D=\min\{wt(c) \mid c \in (C_2 \setminus C_1) \cup (C_1^\perp \setminus C_2^\perp)\}$, where $wt(c)$ is the Hamming weight of $c$.
\end{lemma}

A stabilizer quantum code $C$ is pure if the minimum distance of $C^\bot$ coincides with the minimum Hamming weight of $C^{\bot}\setminus C$.

\begin{theorem} {\rm \cite{10,28}} \label{th:stab}
Let $C$ be an $[N,k,d]_q$-code such that $C\subseteq C^{\bot}$, i.e. $C$ is self-orthogonal. Then there exists an $[[N,N-2k,\geq d^{\bot}]]_q$ stabilizer quantum code, where $d^{\bot}$ denotes the minimum distance of $C^\bot$. If the minimum weight of $C^{\bot}\setminus C$ is equal to $d^{\bot}$, then the stabilizer code is pure and has minimum distance $d^{\bot}$.
\end{theorem}

\begin{corollary}{\rm \cite{SSSSSOCODESSS}}\label{Corollary}
Let $C$ be an $[N,k,d]_q$-code such that $C\subseteq C^\bot$. If  $d>k+1$ then there exists an $[[N,N-2k, d^{\bot}]]_q$-code which is pure.
\end{corollary}
\begin{proof}
$C^\bot$ is an $[N,N-k,d^{\bot}]_q$ code, with $d^{\bot}\leq k+1$ by the Singleton Bound. If $d>k+1$, then by Theorem \ref{th:stab} there exists a pure $[[N,N-2k, d^{\bot}]]_q$ stabilizer quantum code. 
\end{proof}

\subsection{Constructions of AG quantum codes}

We list here some constructions of quantum codes starting from AG codes which have been provided in the literature and exploit self-orthogonality properties of the underlying AG codes.

\begin{itemize}
\item 
\textit{General t-point construction} due to La Guardia and Pereira; see \cite[Theorem 3.1]{LGP}. This is a direct application of the CSS construction to AG codes.

\begin{lemma} \label{lem1} {\rm (General t-point construction)}
Let $\mathcal X$ be a nonsingular curve over $\mathbb F_q$ with genus $g$ and $N+t$ distinct $\mathbb F_q$-rational points, for some $N,t>0$.
Assume that $a_i,b_i$, $i=1,\ldots,t$, are positive integers such that $a_i \leq b_i$ for all $i$ and $2g-2 < \sum_{i=1}^{t} a_i < \sum_{i=1}^t b_i < N$. Then there exists a quantum code with parameters $[[N,k,D]]_{q}$ with $k=\sum_{i=1}^{t} b_i - \sum_{i=1}^{t} a_i$ and $D \geq \min \big\{ N - \sum_{i=1}^{t} b_i, \sum_{i=1}^{t} a_i - (2g-2)\big\}$.
\end{lemma}

\item 
\textit{Quantum codes from weak Castle curves}, due to Munuera, Ten\'orio, and Torres; see \cite[Sections 3.3 and 3.4]{MTT}.

A weak Castle curve over $\mathbb{F}_q$ is a pair $(\mathcal{X},P)$, where $\mathcal{X}$ is an absolutely irreducible $\mathbb{F}_q$-rational curve and $P$ is a rational place of $\mathcal{X}$ such that the following conditions hold.
\begin{itemize}
\item The Weierstrass semigroup $H(P)$ at $P$ is symmetric.
\item there exist a positive integer $s$, a rational map $f:\mathcal{X}\to\mathbb{P}^1$, and a non-empty set $\{\alpha_1,\ldots,\alpha_h\}\subseteq\mathbb{F}_q$ such that $(f)_\infty= sP$ and for all $i=1,\ldots,h$ we have $f^{-1}(\alpha_i)\subseteq\mathcal{X}(\mathbb{F}_q)$ and $|f^{-1}(\alpha_i)=s|$.
\end{itemize}
With the same notation, let $\phi\in\mathbb{F}_q(\mathcal{X})$ be defined as $\phi=\prod_{i=1^h}(f-\alpha_i)$, and let $D$ be the sum of all $N=|\mathcal{X}(\mathbb{F}_q)|-1$ rational places of $\mathcal{X}$ different from $P$.
Denote by $M=\{m_1=0,m_2,\ldots,m_N\}$ the dimension set of $(\mathcal{X},P)$, i.e. $m_i=\min\{m\colon \ell(mP)-\ell((m-N)P)\geq i\}$, and by $C_i$ the weak Castle code $C(D,m_i P)$.
For any $r\geq1$ let $\gamma_r$ be the $r$-th gonality of $\mathcal{X}$, that is the minimum degree of a divisor $A$ on $\mathcal{X}$ such that $\ell(A)\geq r$.
\begin{lemma}{\rm \cite[Corollary 5]{MTT}}
Using the same notation as above, let $(\mathcal{X},P)$ be a weak Castle curve of genus $g$ over $\mathbb{F}_{q^2}$ such that $(d\phi)=(2g-2)P$.
If $(q+1)m_i\leq N+2g-2$ for some $i$, then there exists a quantum code with parameters $[[N,N-2i,\geq d(C_{n-i})]]_q$ with $d(C_{n-i})\geq N-m_{N-i}+\gamma_{a+1}$, where $a=\ell((m_{N-i}-N)P)$.
\end{lemma}

\item \textit{Self-orthogonal AG codes from curves with only one place at infinity}, due to Hernando, McGuire, Monserrat, and Moyano-Fern\'andez; see \cite[Section 3]{SSSSSOCODESSS}.

Let $\mathcal{X}$ be an absolutely irreducible $\mathbb{F}_q$-rational plane curve with $\mathbb{F}_q(\mathcal{X})=\mathbb{F}_q(x,y)$ such that $\mathcal{X}$ has only one point $\mathcal{P}_{\infty}$ at infinity, there is only one place $P_\infty$ centered at $\mathcal{P}_{\infty}$, and $P_{\infty}$ is rational.
Let $\mathcal{A}$ be the set of the elements $a\in\mathbb{F}_q$ such that $\mathcal{X}$ and the line $L_a$ with affine equation $X=a$ are $\mathbb{F}_q$-transversal, that is, the points of $\mathcal{X}\cap L_a$ are $\mathbb{F}_q$-rational and the intersection multiplicity of $\mathcal{X}$ and $L_a$ is $1$ at every point of $\mathcal{X}\cap L_a$.
Let $\mathcal{P}_{\mathcal{A}}$ be the set of places of $\mathcal{X}$ centered at affine points of $\mathcal{X}$ whose $X$-coordinate is in $\mathcal{A}$, and $D$ be the divisor $\sum_{P\in\mathcal{P}_{\mathcal A}}P$.
Define the rational functions $f_{\mathcal A}(x)=\prod_{a\in\mathcal{A}}(x-a)$ and $f_{\mathcal A}^{\prime}(x)$, where $f_{\mathcal A}^\prime (X)=\partial_X f_{\mathcal A}(X)$.
Let $M$ be the divisor of $\mathcal{X}$ such that ${\rm supp}(M)=\{P\in{\rm supp}((f_{\mathcal A}^\prime)_0)\colon P\ne P_{\infty}\}$ and $v_Q(M)=v_Q((f_{\mathcal A}^\prime)_0)$ for every $Q\in{\rm supp}(M)$.
\begin{lemma}{\rm \cite[Theorem 3.1]{SSSSSOCODESSS}}
Using the same notation as above, let $G$ be an $\mathbb{F}_q$-rational divisor of $\mathcal{X}$ with ${\rm supp}(G)\cap{\rm supp}(D)=\emptyset$. Then
$$ C(D,G)^\bot = C(D,(2g-2+\deg(D)-\deg(M))P_{\infty}+M-G) .$$
If in addition $2G\leq(2g-2+\deg(D)-\deg(M))P_{\infty}+M$, then
$$ C(D,G)\subseteq C(D,G)^\bot.$$
\end{lemma}
\end{itemize}

\section{Swiss curves and codes}\label{sec:swiss}
\begin{definition}
A Swiss curve is a pair $(\mathcal{C},P)$ such that $\mathcal{C}$ is an absolutely irreducible $\mathbb{F}_q$-rational curve, $P$ is a place of $\mathbb{F}_q(\mathcal{C})$ and the following holds.
\begin{enumerate}
\item $P$ is rational;
\item there exists a function $x\in \mathbb{F}_q(\mathcal{C})$ such that $(dx)=(2g-2)P$.
\end{enumerate}
\end{definition}

\begin{remark}
Note that the existence of a function $x$ such that $(dx)=(2g-2)P$ implies that the Weierstrass semigroup at $P$ is symmetric, that is, $2g-1 \in G(P)$.
Indeed, $(dx)=(2g-2)P$ implies that $(2g-2)P$ is a canonical divisor and hence the dimension of its Riemann-Roch space is equal to $g$, see \cite[Proposition 1.6.2]{Sti}. Since there are exactly $g$ elements in $G(P)$ (and they are at most $2g-1$) we get that $2g-1 \in G(P)$.
\end{remark}

Even though Condition $(1)$ is not difficult to be forced, Condition $(2)$ seems to be quite cryptic. The following remark describes a way to force also Condition $(2)$ to hold.

\begin{remark}\label{remo}
One way to force the existence of $x$ is the following. Suppose that there exists a function $x \in \mathbb{F}_q(\mathcal{C})$ such that $dx \ne 0 $ and in $\mathbb{F}_q(\mathcal{C}) / \mathbb{F}_q(x)$ there is a unique ramification place and it is totally ramified. Without loss of generality we can assume that the totally ramified place is the pole $P_\infty$ of $x$. In fact, if such a place is the zero of $x-\alpha$, it is enough to replace $x$ with $1/(x-\alpha)$ and consider $\mathbb{F}_q(\mathcal{C}) / \mathbb{F}_q(1/(x-\alpha))$.
From \cite[Theorem 3.4.6]{Sti}, 
$$({\rm{Cotr}}_{\mathbb{F}_q(\mathcal{C}) / \mathbb{F}_q(x)}(dx))=(dx)_{\mathbb{F}_q(\mathcal{C})}={\rm{Con}}_{\mathbb{F}_q(\mathcal{C}) / \mathbb{F}_q(x)}((dx))+{\rm{Diff}}(\mathbb{F}_q(\mathcal{C}) / \mathbb{F}_q(x)).$$
 Since the support of both ${\rm{Con}}_{\mathbb{F}_q(\mathcal{C}) / \mathbb{F}_q(x)}((dx))$ and ${\rm{Diff}}(\mathbb{F}_q(\mathcal{C}) / \mathbb{F}_q(x))$  is just  $P_\infty$, we get that $(dx)_{\mathbb{F}_q(\mathcal{C})}=(2g-2)P_\infty$.
\end{remark}

Swiss curves can be constructed as explained in the following remark.

\begin{remark} \label{remo1}
Let $\mathcal{C}$ be an $\mathbb{F}_q$-rational curve of $p$-rank zero. Assume that there exist a rational place $P$  of $\mathbb{F}_q(\mathcal{C})$ and a $p$-subgroup $S$ of automorphisms of $\mathbb{F}_q(\mathcal{C})$ fixing $P$ such that the quotient curve $\mathcal{C}/S$ is rational. Then $(\mathcal{C},P)$ is a Swiss curve. Indeed $(1)$ is trivially satisfied and from \cite[Lemma 11.129]{HKT} $P$ is the unique place ramifying in $\mathcal{C}/S$ and it is totally ramified. Hence also Condition $(2)$ is satisfied by Remark \ref{remo}.
 \end{remark}

In the following, we will denote by $\mathbb{P}_q$ the set of all rational places of $ \mathbb{F}_q(\mathcal{C})$. Also, given a divisor $D$ and a place $Q$, we denote by $v_Q(D)$ the weight of $D$ at $Q$.

Consider a Swiss curve $(\mathcal{C},P)$ and the set
$$\mathcal{A}=\left\{\alpha \in  \mathbb{F}_q \ :\ (x-\alpha)_0-v_P((x-\alpha)_0) P\leq \sum_{Q\in \mathbb{P}_q\setminus  \{P\}} Q \right\}.$$

Basically, $\mathcal{A}$ consists of all the $\alpha\in \mathbb{F}_q$ such that all the zeros of the function $x-\alpha$ other than (possibly) $P$ are rational and simple.

Also, let 
$$D=\sum_{\alpha \in \mathcal{A}} \bigg((x-\alpha)_0-v_P((x-\alpha)_0) P \bigg).$$

\begin{theorem}\label{Main:Swiss}
Let $(\mathcal{C},P)$ be a Swiss curve. With the same notation as above, consider   another $\mathbb{F}_q$-rational divisor $G$ such that ${\rm supp}(G)\cap {\rm supp}(D) =\emptyset$. Then
\begin{enumerate}
    \item $C(D,G)^{\bot} =C(D,E+(\gamma+2g-2) P-G)$, for some positive divisor $E$ and some integer $\gamma$;
    \item if, in addition, $2G\leq E+(\gamma+2g-2) P $ then $C(D,G)\subset C(D,G)^{\bot}$.
\end{enumerate}
\end{theorem}
\begin{proof}
Define
$$h=\sum_{a \in \mathcal{A}} \frac{1}{x-a},\qquad\omega=\left(h\right)dx.$$
Clearly, places in ${\rm supp}(D)$ are simple poles of $h$. 

By hypothesis $(dx)=(2g-2)P$. Also, $\left(h\right)=E-D+\gamma P$, 
where
$$E=\left(h\right)_0-v_P((h)_0) P,\qquad\gamma =\deg D-\deg E.$$
Hence, $(\omega)=E-D+(\gamma+2g-2) P$.

Therefore $\omega$ has poles at places of $D$ and it is readily seen that the residue of $\omega$ at such places is $1$.
Now, the claim follows from \cite[Theorem 2.72]{27}.
\end{proof}

In Section \ref{sec:appl}, we describe several Swiss curves. Using Theorem \ref{Main:Swiss} we construct families of self-orthogonal codes, which provide stabilizer quantum codes by means of Theorem \ref{th:stab}.

\section{Applications to some Swiss curves}\label{sec:appl}

\subsection{GK curve}
The Giulietti-Korchm\'aros curve  over $\mathbb{F}_{q^6}$ is a non-singular curve in ${\rm PG}(3,\mathbb{K})$, $\mathbb{K}=\overline{\mathbb{F}}_{q}$, defined by the affine equations:
\[GK_q:
\begin{cases}
Y^{q+1} = X^q+X,\\
Z^{q^2-q+1} =Y^{q^2}-Y.
\end{cases}
\]
It has genus $g=\frac{(q^3+1)(q^2-2)}{2}+1$, and the number  of its $\mathbb{F}_{q^6}$-rational places is $q^8-q^6+q^5+1$. The GK curve first appeared in~\cite{GK2009} as a maximal curve over $\mathbb{F}_{q^6}$, since the latter number coincides with the Hasse-Weil upper bound, $q^6+2gq^3+1$.
The GK curve is the first example of an $\mathbb{F}_{q^6}$-maximal curve that is not $\mathbb{F}_{q^6}$-covered by the Hermitian curve, provided that $q>2$.

Since this curve is $\mathbb{F}_{q^6}$-maximal its $p$-rank is zero. Indeed, we will show that Condition $(2)$ is satisfied by applying Remark \ref{remo1}. 
The coordinate function $z$ has valuation $1$ at each affine $\mathbb{F}_{q^6}$-rational point of $GK_q$, hence $z$ is a separating element for $\mathbb{K}(GK_q)/\mathbb{K}$ by \cite[Prop. 3.10.2]{Sti}. Then $dz$ is non-zero by \cite[Prop. 4.1.8(c)]{Sti}. It is easily checked that $\mathbb{F}_{q^6}(GK_q)/\mathbb{F}_{q^6}(z)$ is a Galois extension of degree $q^3$; also, the unique place $P_\infty$ centered at the unique point at infinity of $GK_q$ is a ramification place for $\mathbb{F}_{q^6}(GK_q)/\mathbb{F}_{q^6}(z)$.
From Remark \ref{remo1}, $(GK_q,P_\infty)$ is a Swiss curve and
$(dz)=(2g-2)P_{\infty}=(q^3+1)(q^2-2)P_\infty$.

Let $m=q^2-q+1$.
It can be seen that if $\xi \in \mathbb{F}_{q^6}$ is such that $Y^{q^2}-Y=\xi^m$ has $q^2$ solutions in $\mathbb{F}_{q^6}$, then for each $\eta\in \mathbb{F}_{q^6}$ satisfying $\eta^{q^2}-\eta=\xi^m$ there are precisely $q$ values $\theta \in \mathbb{F}_{q^6}$ such that $\theta^q+\theta = \eta^{q+1}$. 
Also, the values $\xi\in \mathbb{F}_{q^6}$ for which all the zeros of $z-\xi$  are rational are those satisfying 
\begin{equation}\label{eq:cond}
\xi^{mq^4}+\xi^{mq^2}+\xi^m=0;
\end{equation}
moreover for each of them there are exactly $q^3$ triples $(\bar x,\bar y,\bar z) \in \mathbb{F}_{q^6}^3$ such that 
$\bar{z}^m = \bar{y}^{q^2}-\bar{y}$ and $\bar{y}^{q+1}=\bar{x}^q+\bar{x}$. This implies that there are exactly $q^5-q^3+q^2$ values $\xi\in \mathbb{F}_{q^6}$ satisfying Equation \eqref{eq:cond}.
Let 
$$\Xi =\{\xi \in \mathbb{F}_{q^6} \ : \  \xi^{mq^4}+\xi^{mq^2}+\xi^m=0\}.$$
Then $$\Xi \setminus \{0\}=\left\{\xi \in \mathbb{F}_{q^6} \ : \  \left(\xi^{(q-1)(q^3+1)}\right)^{q^2+1}+\left(\xi^{(q-1)(q^3+1)}\right)+1=0\right\}.$$
Note that if $\mu_{q^2+q+1}$ denotes the set of the $(q^2+q+1)$-th roots of unity then
$$\{ \theta \in \mu_{q^2+q+1} \ : \ \theta^{q^2+1}+\theta+1 =0\}=\{ \theta \in \mu_{q^2+q+1} \ : \ \theta^{q+1}+\theta^q+1 =0\}.$$
Thus, the polynomial 
$$f(Z)=Z^{q^5-q^3+q^2}+Z^{q^5-q^4+q^2-q+1}+Z\in \mathbb{F}_{q^6}[Z]$$ factorizes completely over $\mathbb{F}_{q^6}$, and
$$f(Z)=\prod _{\xi \in \Xi }(Z-\xi).$$
Also, $$f^{\prime}(Z)
=Z^{q^5-q^4+q^2-q}+1=(Z^{(q^3+1)(q-1)}+1)^q,$$
and hence the zero divisor of the rational function $f^\prime(z)\in\mathbb{K}(GK_q)$ satisfies
$$\deg(f^{\prime}(z))_0 =(q^5-q^4+q^2-q)q^3.$$
Now consider in $\mathbb{K}(GK_q)$ the function 
$$ \sum_{\xi \in \Xi}\frac{1}{z-\xi}=\frac{f^{\prime}(z)}{f(z)}.$$
Its principal divisor is 
$$M-D+(q^4-q^3+q)q^3P_{\infty},$$
where $M$ is the zero divisor of $f^{\prime}(z)$  and $D$ is the zero divisor $\sum_{P \in \mathbb{P}_{q^6}(GK_q)\setminus \{P_{\infty}\}}P$ of $f(z)$ of degree $q^5(q^3-q+1)$.
The divisor of 
$$\omega=\sum_{\xi \in \Xi}\frac{1}{z-\xi}\,dz$$
is 
$$M-D+[2g-2+(q^4-q^3+q)q^3]P_{\infty}=M-D+[q^7-q^6+q^5+q^4-2q^3+q^2-2]P_{\infty}.$$

Consider the one-point divisor $G=sP_{\infty}$. 
By Theorem \ref{Main:Swiss} and its proof,
\begin{eqnarray*}
C(D,G)^{\bot}&=&C(D,M+[2g-2+(q^4-q^3+q)q^3-s]P_{\infty})\\
&=&C(D,M+[q^7-q^6+q^5+q^4-2q^3+q^2-2-s]P_{\infty}).
\end{eqnarray*}
Also, $C(D,G) \subset C(D,G)^{\bot}$ if
$$s \leq \frac{q^7-q^6+q^5+q^4-2q^3+q^2-2}{2}.$$
Finally, by Theorem \ref{th:stab}, we obtain the following result.
\begin{theorem}\label{Th:GKQuantum}
With the same notation as above, consider the $q^6$-ary code $C(D,s P_{\infty})$ from the GK curve.
Assume that 
$$q^5-2q^3+q^2-2\leq s \leq \frac{q^7-q^6+q^5+q^4-2q^3+q^2-2}{2}.$$
Then there exists
a quantum code with parameters
$$[[\,q^8-q^6+q^5,\;q^8-q^6+2q^5-2q^3+q^2-2-2s,\;\geq s-q^5+2q^3-q^2+2\,]]_{q^6}.$$
\end{theorem}

\subsection{GGS curves}\label{SubSection:GGS}
Let $q$ be a prime power and $n\geq5$ be an odd integer. 
The GGS curve $GGS(q,n)$ is defined by the equations
\begin{equation}\label{GGS_equation}
GGS(q,n): \left\{
 \begin{array}{l}
X^q + X = Y^{q+1}\\
Y^{q^2}-Y= Z^m\\
\end{array}
\right.
,
\end{equation}
where $m= (q^n+1)/(q+1)$; see \cite{GGS2010}.
The genus of $GGS(q,n)$ is $\frac{1}{2}(q-1)(q^{n+1}+q^n-q^2)$, and $GGS(q,n)$ is $\mathbb F_{q^{2n}}$-maximal.
Let $P_0=(0,0,0)$, $P_{(a,b,c)}=(a,b,c)$, and let $P_{\infty}$ be the unique ideal point of $GGS(q,n)$.
Note that $GGS(q,n)$ is singular, being $P_\infty$ its unique singular point. Yet, there is only one place of $GGS(q,n)$ centered at $P_\infty$.
The divisors of the coordinate functions $x,y,z$ satisfying $x^q + x = y^{q+1}$ and $y^{q^2}-y= z^m$ are
\begin{eqnarray*}
(x)&=&m(q+1)P_0-m(q+1)P_{\infty},\\
(y)&=&m\sum_{\alpha^q+\alpha=0} P_{(\alpha,0,0)}-mqP_{\infty},\\
(z)&=&\sum_{\scriptsize\begin{array}{l} \alpha^q+\alpha=\beta\\ \beta \in \mathbb{F}_{q^2}\\ \end{array}} P_{(\alpha,\beta,0)}-q^3P_{\infty}.
\end{eqnarray*}
As for the GK curve, the curve $GGS(q,n)$ has $p$-rank zero because it is $\mathbb{F}_{q^{2n}}$-maximal, and $(dz)=(2g-2)P_\infty$ being $\mathbb F_{q^{2n}}(GGS(q,n))/\mathbb F_{q^{2n}}(z)$ a Galois extension of degree $q^3$ in which $P_\infty$ is totally ramified. Hence $(GGS(q,n),P_\infty)$ is a Swiss curve.
From the proof of \cite[Theorem 2.6]{GGS2010} every $\mathbb{F}_{q^{2n}}$-rational point of the curve $Y^{q^2}-Y=Z^m$ which is not centered at the unique point at infinity of the curve splits completely in $\mathbb{F}_{q^{2n}}(GGS(q,n))/\mathbb{F}_{q^{2n}}(y,z)$. This is equivalent to say, as for the GK curve, that if $\xi \in \mathbb{F}_{q^{2n}}$ is such that $Y^{q^2}-Y=\xi^m$ has $q^2$ solutions in $\mathbb{F}_{q^{2n}}$, then for each $\eta\in \mathbb{F}_{q^{2n}}$ satisfying $\eta^{q^2}-\eta=\xi^m$ there are precisely $q$ values $\theta \in \mathbb{F}_{q^{2n}}$ such that $\theta^q+\theta = \eta^{q+1}$. 
Also, the values $\xi\in \mathbb{F}_{q^{2n}}$ for which all  the zeros of $z-\xi$ belong to $\mathbb{F}_{q^{2n}}$ are those satisfying 
$$\sum_{i=0}^{n-1} (\xi^m)^{q^{2i}}=0;$$
moreover for each of them there are exactly $q^3$ triples $(x,y,z) \in \mathbb{F}_{q^{2n}}^3$ such that 
$z^m = y^{q^2}-y$ and $y^{q+1}=x^q+x$. 
This means that there are exactly $q^{2n-1}-q^n+q^{n-1}$ such values $\xi\in \mathbb{F}_{q^{2n}}$ as $|GGS(q,n)(\mathbb{F}_{q^{2n}})|=q^{2n+2}-q^{n+3}+q^{n+2}+1$.
Let 
$$\Xi =\{\xi \in \mathbb{F}_{q^{2n}} \ : \  \sum_{i=0}^{n-1} (\xi^m)^{q^{2i}}=0\}.$$
Then $$\Xi \setminus \{0\}=\{\xi \in \mathbb{F}_{q^{2n}} \ : (\xi^m)^{q^{2(n-1)}-1}+(\xi^m)^{q^{2(n-2)}-1}+\cdots+(\xi^m)^{q^2-1}+1=0\}$$
has cardinality $(q^n+1)(q^{n-1}-1)$.

Let $\mu_{(q^n-1)/(q-1)}$ be the set of $\frac{q^n-1}{q-1}$-th roots of unity, and let $k=\frac{n-1}{2}\geq2$. Then
$$\Theta=\{\theta \in \mu_{(q^n-1)/(q-1)} \mid \theta^{q^{2(n-2)}+q^{2(n-3)}+\cdots+q^2+1} + \theta^{q^{2(n-3)}+\cdots+q^2+1}+\cdots+\theta^{q^2+1}+\theta+1=0\}$$
$$=\{\theta \in \mu_{(q^n-1)/(q-1)} \mid p(\theta)=0\},$$
where
\begin{equation}\label{eq:polp}
p(Z)=1+\sum_{i=0}^{k-1} Z^{\sum_{j=0}^{i}q^{2j} + \sum_{j=0}^{k-1} q^{2j+1}} + \sum_{i=0}^{k-1} Z^{\sum_{j=0}^{i} q^{2j+1}}\;\in\mathbb{F}_{q^{2n}}[Z]
\end{equation}
which is a separable polynomial of degree $\sum_{j=0}^{2k-1} q^j$, see the proof of \cite[Lemma 2]{ABQ} and in particular \cite[Equation (4)]{ABQ}.

Thus, the polynomial 
$$f(Z)=Z \cdot p(Z^{(q^n+1)(q-1)}) \in \mathbb{F}_{q^{2n}}[X]$$ factorizes completely over $\mathbb{F}_{q^{2n}}$, and
\begin{eqnarray*}
f(Z)&=\prod _{\xi \in \Xi }(Z-\xi)=&Z+\sum_{i=0}^{k-1} Z^{1+\sum_{j=0}^{i} q^{2j}(q^n+1)(q-1) + \sum_{j=0}^{k-1} q^{2j+1} (q^n+1)(q-1)}\\
& & + \sum_{i=0}^{k-1} Z^{1+\sum_{j=0}^{i} q^{2j+1}(q^n+1)(q-1)}.
\end{eqnarray*}
Also, $$f^{\prime}(Z)=1+Z^{q(q^n+1)(q-1)} + \sum_{i=1}^{k-1} Z^{\sum_{j=0}^{i} q^{2j+1}(q^n+1)(q-1)}$$
and hence 
$$\deg(f^{\prime}(z))_0 =\bigg(q\frac{(q^n+1)}{q+1}(q^{n-1}-1)\bigg)q^3.$$
Now consider the function 
$$ \sum_{\xi \in \Xi}\frac{1}{z-\xi}=\frac{f^{\prime}(z)}{f(z)}.$$
Its principal divisor is 
$$M-D+\bigg( (q^{n-1}-1)\frac{q^n+1}{q+1}+1\bigg)q^3P_{\infty},$$
where $M$ is the zero divisor of $f^{\prime}(z)$  and 
$$D=(f(z))_0= \sum_{P \in \mathbb{P}_{q^{2n}}(GGS(q,n))\setminus \{P_{\infty}\}}P$$ has degree $q^3((q^{n-1}-1)(q^n+1)+1)= q^{2n+2}-q^{n+3}+q^{n+2}$.
The principal divisor of 
$$\omega=\sum_{\xi \in \Xi}\frac{1}{z-\xi}\,dz$$
is 
$$M-D+\bigg[\bigg( (q^{n-1}-1)\frac{q^n+1}{q+1}+1\bigg)q^3+2g-2\bigg]P_{\infty}.$$
Consider the one-point divisor $G=s P_{\infty}$.
By Theorem \ref{Main:Swiss} and its proof,
$$
C(D,G)^{\bot}=C\bigg(D,M+\bigg[\bigg( (q^{n-1}-1)\frac{q^n+1}{q+1}+1\bigg)q^3+2g-2-s\bigg]P_{\infty}\bigg)$$
$$=C\bigg(D,M+\bigg[\bigg( (q^{n-1}-1)\frac{q^n+1}{q+1}+1\bigg)q^3+(q-1)(q^{n+1}+q^n-q^2)-2-s\bigg]P_{\infty}\bigg).$$
Also, $C(D,G) \subset C(D,G)^{\bot}$ if
$$s \leq \frac{\bigg[\bigg( (q^{n-1}-1)\frac{q^n+1}{q+1}+1\bigg)q^3+(q-1)(q^{n+1}+q^n-q^2)-2\bigg]}{2}.$$
From Theorem \ref{th:stab} we have the following result.
\begin{theorem}\label{Th:GGSQuantum}
With the same notation as above, consider the $q^{2n}$-ary code $C(D,mP_{\infty})$ from the GGS curve.
Assume that 
$$(q-1)(q^{n+1}+q^n-q^2)-2\leq s \leq \frac{\bigg[\bigg( (q^{n-1}-1)\frac{q^n+1}{q+1}+1\bigg)q^3+(q-1)(q^{n+1}+q^n-q^2)-2\bigg]}{2}.$$
Then there exists
a quantum code with parameters
$$[[\,q^{2n+2}-q^{n+3}+q^{n+2},\;q^{2n+2}-q^{n+3}+q^{n+2}+(q-1)(q^{n+1}+q^n-q^2)-2-2s,$$
$$\geq s-(q-1)(q^{n+1}+q^n-q^2)+2\,]]_{q^{2n}}.$$
\end{theorem}

\subsection{Abd\'on-Bezerra-Quoos curve}\label{SubSection:ABQ}

Let $q$ be a prime power and $n\geq3$ be an odd integer. 
The Abd\'on-Bezerra-Quoos curve $ABQ(q,n)$ is defined by the equation
\begin{equation}\label{ABQ_equation}
ABQ(q,n): 
Y^{q^2}-Y= X^m,
\end{equation}
where $m= (q^n+1)/(q+1)$; see \cite{ABQ,GGS2010}.
The curve $ABQ(q,n)$ is singular, has genus $\frac{1}{2}(q-1)(q^{n}-q)$, and is $\mathbb F_{q^{2n}}$-maximal.
Let $P_0=(0,0,0)$, $P_{(a,b,c)}=(a,b,c)$, and let $P_{\infty}$ be the unique ideal point of $ABQ(q,n)$.
The point $P_\infty$ is the unique singular point of $ABQ(q,n)$. Yet, there is only one place of $ABQ(q,n)$ centered at $P_\infty$.
As for the GK and GGS cases, $ABQ(q,n)$  is $\mathbb{F}_{q^{2n}}$-maximal and hence has $p$-rank zero. The extension  $\mathbb F_{q^{2n}}(ABQ(q,n))/\mathbb F_{q^{2n}}(x)$ is a Galois extension of degree $q^2$, and $P_\infty$ is totally ramified in it. Thus, $(dx)=(2g-2)P_\infty$ and $(ABQ(q,n),P_\infty)$ is a Swiss curve.

An element  $\xi \in \mathbb{F}_{q^{2n}}$ is such that $Y^{q^2}-Y=\xi^m$ has $q^2$ solutions in $\mathbb{F}_{q^{2n}}$ if and only if
$$\sum_{i=0}^{n-1} (\xi^m)^{q^{2i}}=0.$$
Also, there are exactly $q^{2n-1}-q^n+q^{n-1}$ such values $\xi\in \mathbb{F}_{q^{2n}}$ as $|ABQ(q,n)(\mathbb{F}_{q^{2n}})|=q^{2n+1}-q^{n+2}+q^{n+1}+1$.

Arguing as in Section \ref{SubSection:GGS}, the polynomial
$$f(Z)=Z \cdot p(Z^{(q^n+1)(q-1)})\, \in \mathbb{F}_{q^{2n}}[Z]$$ factorizes completely over $\mathbb{F}_{q^{2n}}$; here, the polynomial $p(X)\in\mathbb{F}_{q^{2n}}[X]$ is as in Equation \eqref{eq:polp}. Also,
\begin{eqnarray*}
f(Z)=\prod _{\xi \in \Xi }(Z-\xi)&=&Z+\sum_{i=0}^{k-1} Z^{1+\sum_{j=0}^{i} q^{2j}(q^n+1)(q-1) + \sum_{j=0}^{k-1} q^{2j+1} (q^n+1)(q-1)} \\
&&+ \sum_{i=0}^{k-1} Z^{1+\sum_{j=0}^{i} q^{2j+1}(q^n+1)(q-1)}.
\end{eqnarray*}
Also, $$f^{\prime}(Z)=1+Z^{q(q^n+1)(q-1)} + \sum_{i=1}^{k-1} Z^{\sum_{j=0}^{i} q^{2j+1}(q^n+1)(q-1)}$$
and hence 
$$\deg(f^{\prime}(z))_0 =\bigg(q\frac{(q^n+1)}{q+1}(q^{n-1}-1)\bigg)q^2.$$
Now, the function 
$$ \sum_{\xi \in \Xi}\frac{1}{z-\xi}=\frac{f^{\prime}(z)}{f(z)}$$
has principal divisor 
$$M-D+\bigg( (q^{n-1}-1)\frac{q^n+1}{q+1}+1\bigg)q^2 P_{\infty},$$
where $M$ is the zero divisor of $f^{\prime}(z)$  and 
$$D=(f(z))_0= \sum_{P \in \mathbb{P}_{q^{2n}}(ABQ(q,n))\setminus \{P_{\infty}\}}P$$ has degree $q^2((q^{n-1}-1)(q^n+1)+1)= q^{2n+1}-q^{n+2}+q^{n+1}$.
The principal divisor of 
$$\omega=\sum_{\xi \in \Xi}\frac{1}{z-\xi}dz$$
is 
$$M-D+\bigg[\bigg( (q^{n-1}-1)\frac{q^n+1}{q+1}+1\bigg)q^2+2g-2\bigg]P_{\infty}.$$
Consider the one-point divisor $G=s P_{\infty}$. By Theorem \ref{Main:Swiss} and its proof,
\begin{eqnarray*}
C(D,G)^{\bot}&=&C\bigg(D,M+\bigg[\bigg( (q^{n-1}-1)\frac{q^n+1}{q+1}+1\bigg)q^2+2g-2-s\bigg]P_{\infty}\bigg)\\
&=&C\bigg(D,M+\bigg[\bigg( (q^{n-1}-1)\frac{q^n+1}{q+1}+1\bigg)q^2+(q-1)(q^n-q)-2-s\bigg]P_{\infty}\bigg).
\end{eqnarray*}
Also, $C(D,G) \subset C(D,G)^{\bot}$ if
$$s \leq \frac{\bigg[\bigg( (q^{n-1}-1)\frac{q^n+1}{q+1}+1\bigg)q^2+(q-1)(q^{n}-q)-2\bigg]}{2}.$$
The theorem below follows from Theorem \ref{th:stab}.
\begin{theorem}\label{Th:ABQQuantum}
With the same notation as above, consider the $q^{2n}$-ary code $C(D,mP_{\infty})$ from the ABQ curve.
Assume that 
$$(q-1)(q^{n}-q)-2\leq s \leq \frac{\bigg[\bigg( (q^{n-1}-1)\frac{q^n+1}{q+1}+1\bigg)q^2+(q-1)(q^{n}-q)-2\bigg]}{2}.$$
Then there exists a quantum code with parameters
$$[[\,q^{2n+1}-q^{n+2}+q^{n+1},\;q^{2n+1}-q^{n+2}+q^{n+1}+(q-1)(q^{n}-q)-2-2s,\,\geq s-(q-1)(q^{n}-q)+2\,]]_{q^{2n}}.$$
\end{theorem}

\subsection{Suzuki and Ree curves}
Let $q_0= 2^s$, where $s\geq1$, and $q= 2q_0^2$. The Suzuki curve $S_q$ is given by the affine model 
$$S_q: Y^q+Y=X^{q_0}(X^q+X).$$
 The curve $S_q$  is $\mathbb{F}_{q^4}$-maximal of genus $q_0(q-1)$. It has a unique singular point, namely its unique point at infinity $P_\infty$,which is a $q_0$-fold point and the center of just one place of $S_q$.
The extension $\mathbb{F}_{q^4}(S_q)/\mathbb{F}_{q^4}(x)$ is a Galois extension of degree $q$ in which $P_\infty$ is the only ramified place, and it is totally ramified. Hence, by Remark \ref{remo1}, $(dx)=(2g-2)P_\infty=(2q_0(q-1)-2)P_\infty$ and $(S_q,P_\infty)$ is a Swiss curve.

Let $q_0= 3^s$, where $s\geq1$, and $q= 3q_0^2$. The Ree curve $R_q$ is given by the affine space model 
$$R_q:\begin{cases} Y^q-Y=X^{q_0}(X^q-X), \\ Z^q-Z=X^{2q_0}(X^q-X)\end{cases}.$$
This curve has genus $\frac{3}{2}q_0(q-1)(q+q_0+1)$ and it is $\mathbb{F}_{q^6}$-maximal. It has a unique singular point coinciding with its unique infinite point; moreover there is a unique place $P_\infty$ centered in it. The extension $\mathbb{F}_{q^6}(R_q)/\mathbb{F}_{q^6}(x)$ is a Galois extension of degree $q^2$ in which $P_\infty$ is the only ramified place, and it is totally ramified. Hence, by Remark \ref{remo1}, $(dx)=(2g-2)P_\infty$ and $(R_q,P_\infty)$ is a Swiss curve.

\begin{remark} Since Suzuki and Ree curves are Swiss curves, it makes sense to ask for a suitable set $I$ of rational points as well as a covering of $I$ made of lines, to which Theorem \ref{Main:Swiss} applies. According to the equations defining the curves, the most natural choice would probably be $I={S}_q(\mathbb{F}_q)$ and $I={R}_q(\mathbb{F}_q)$ respectively. Indeed, in both cases a nice covering of lines is given simply by the vertical lines $x=a$, with $a \in \mathbb{F}_q$.
However, in this case one would obtain $f(X)=\prod_{a\in\mathbb{F}_q}(X-a)=X^q-X$, which has clearly constant derivative. Hence the construction would be the same as in \cite{MTT}. The determination of a suitable set $I$ and a covering of lines remains an open problem.
\end{remark}

\section{$r$-Swiss curves and codes}
\label{sec:r-swiss}

In this section we generalize the construction of Section \ref{sec:swiss} to a larger class of curves.
\begin{definition}
Let $r$ be a positive integer. An $r$-Swiss curve is an $(r+1)$-tuple $(\mathcal{C},P_1,\ldots,P_r)$ such that $\mathcal{C}$ is an absolutely irreducible $\mathbb{F}_q$-rational curve, $P_1,\ldots,P_r$ are distinct places of $\mathbb{F}_q(\mathcal{C})$, and the following properties hold:
\begin{enumerate}
\item $P_i$ is rational for every $i=1,\ldots,r$;
\item there exists a function $x\in \mathbb{F}_q(\mathcal{C})$ such that $(dx)=\frac{2g-2}{r}(P_1+\ldots+P_r)$;
\item ${\rm supp}\left((x)_\infty\right)=\{P_1,\ldots,P_r\}$.
\end{enumerate}
\end{definition}

\begin{remark}
Clearly, a $1$-Swiss curve is just a Swiss curve.
\end{remark}
Consider an $r$-Swiss curve $(\mathcal{C},P_1,\ldots,P_r)$ and the set
$$\mathcal{A}=\left\{\alpha \in  \mathbb{F}_q \ :\ (x-\alpha)_0- \sum_{i=1}^r v_{P_i}((x-\alpha)_0) P_i\leq \sum_{Q\in \mathbb{P}_q\setminus  \{P_1,\ldots,P_r\}} Q \right\}.$$
The set $\mathcal{A}$ consists of all elements $\alpha\in \mathbb{F}_q$ such that all the zeros of the function $x-\alpha$ other than (possibly) $P_1,\ldots,P_r$ are rational and simple.
Also, let 
$$D=\sum_{\alpha \in \mathcal{A}}\left(  (x-\alpha)_0- \sum_{i=1}^r v_{P_i}((x-\alpha)_0) P_i \right).$$

\begin{theorem}\label{Main:WSwiss}
Let $(\mathcal{C},P_1,\ldots, P_r)$ be an $r$-Swiss curve. With the same notation as above, consider   another $\mathbb{F}_q$-rational divisor $G$ such that $\rm{supp}(G)\cap \rm{supp}(D) =\emptyset$. Then
\begin{enumerate}
    \item $C(D,G)^{\bot} =C(D,E+\sum_{i=1}^r(\gamma_i+2g-2) P_i-G)$, for some positive divisor $E$ and some integers $\gamma_1,\ldots,\gamma_r$;
    \item if, in addition, $2G\leq E+\sum_{i=1}^r\left(\gamma_i+\frac{2g-2}{r}\right) P_i $ then $C(D,G)\subset C(D,G)^{\bot}$.
\end{enumerate}
\end{theorem}
\begin{proof}
Let $h=\sum_{a \in \mathcal{A}} \frac{1}{x-a}$. Clearly, places in ${\rm supp}(D) \setminus \{P_1,\ldots, P_r\}$ are simple poles of $h$. 
Consider 
$$\omega=\left(h \right) dx.$$
By hypothesis $(dx)=\sum_{i=1}^r\frac{2g-2}{r}P_i$, and 
$$\left(h\right)=E-D+ \sum_{i=1}^r\gamma_i P_i,$$
where $E=\left(h\right)_0-\sum_{i=1}^r v_{P_i}((h)_0)P_i$ and $\sum_{i=1}^r \gamma_i =\deg D-\deg E$. Summing up,
$$(\omega)=Z-D+\sum_{i=1}^r\bigg(\gamma_i+\frac{2g-2}{r}\bigg) P_i.$$
Therefore $\omega$ has poles at places of $D$ and it is readily seen that the residue of $\omega$ at such places is $1$.
Now the claim follows from \cite[Theorem 2.72]{27}.
\end{proof}

\section{Applications to some $r$-Swiss curves}\label{sec:r-appl}

\subsection{GGK curves}\label{SubSection:GGK}
Let $q$ be a prime power and $n\geq3$ be an odd integer. 
The  curve $GGK2(q,n)$ is defined by the equations
\begin{equation}\label{GGK_equation}
GGK2(q,n): \left\{
 \begin{array}{l}
X^{q+1} -1 = Y^{q+1}\\
Y\bigg(\frac{X^{q^2}-X}{X^{q+1}-1}\bigg)= Z^m\\
\end{array}
\right.
,
\end{equation}
where $m= (q^n+1)/(q+1)$; see \cite{BM}.
The genus of $GGK2(q,n)$ is $\frac{1}{2}(q-1)(q^{n+1}+q^n-q^2)$, $GGK2(q,n)$ is $\mathbb F_{q^{2n}}$-maximal, and $GGK2(q,3)\cong GK_q$. 
The coordinate function $x$ has exactly $q+1$ distinct poles $P_1,\ldots,P_{q+1}$; also, the coordinate function $z$ has pole divisor $(z)_\infty=(q^2-q)(P_1+\ldots+P_{q+1})$, and $(dz)=\frac{2g-2}{q+1}(P_1+\ldots+P_{q+1})$ (see \cite[Page 17]{BM}). Hence, $(GGK2(q,n),P_1,\ldots,P_{q+1})$ is a $(q+1)$-Swiss curve.

\subsubsection{\bf The case $q=2$.}
In the rest of this section, $q=2$ and $GGK2(q,n)$ reads
\begin{equation*}
GGK2(2,n): \left\{
 \begin{array}{l}
X^{3} -1 = Y^{3}\\
YX= Z^{(2^n+1)/3}\\
\end{array}
\right..
\end{equation*}
It is easily seen that $z=a$ has exactly $q^3-q=6$ rational zeros if and only if either $a=0$ or $Y^6+Y^3-a^{2^n+1} \in \mathbb{F}_{2^{2n}}[Y]$ has $6$ distinct roots in $\mathbb{F}_{2^{2n}}$. 
From the maximality of $GGK2(2,n)$ and $[\mathbb{F}_{2^{2n}}(x,y,z) : \mathbb{F}_{2^{2n}}(z)]=6$ follows that the set
 $$\mathcal{A}=\{a \in \mathbb{F}_{2^{2n}}^* \mid Y^6+Y^3=a^{2^n+1} \ \textrm{has 6 distinct roots in} \ \mathbb{F}_{2^{2n}}\}$$
has size
$$|\mathcal{A}|=4(2^n+1)(2^{n-1}-1)/3.$$
Let $f(X)=\prod_{a \in \mathcal{A}}(x-a)$.
The following can be checked by direct computation with MAGMA. 
\begin{itemize}
\item $n=3$. In this case,
$$f(X)=X^{36} + X^{27} + X^{18} + 1, \qquad f^\prime (X)=x^{26}.$$
Since $P_1,P_2,P_3$ are not zeros of $f^\prime$, we have $(f^\prime(z))_\infty=26(q^2-q)(P_1+P_2+P_3)$ and hence
$$\deg(f^{\prime}(z))_0 =26(q+1)(q^2-q)=156.$$
Now, the function 
$$ \sum_{\xi \in \mathcal{A}}\frac{1}{z-\xi}=\frac{f^{\prime}(z)}{f(z)}$$
has principal divisor 
$$M-D+20 \sum_{i=1}^{3}P_i,$$
where $M$ is the zero divisor of $f^{\prime}(z)$  and 
$$D=(f(z))_0= \sum_{P \in \mathbb{P}_{2^{6}}(GGK2(2,3))\setminus \mathbb{P}_{2^{2}}(GGK2(2,3)) }P$$ has degree $216$.
The principal divisor of 
$$\omega=\sum_{\xi \in \mathcal{A}}\frac{1}{z-\xi}dz$$
is 
$$M-D+26 \sum_{i=1}^{3}P_i.$$
Consider the multi-point divisor $G=s \sum_{i=1}^{3} P_i$.
By Theorem \ref{Main:Swiss} and its proof,
$$C(D,G)^{\bot}=C\left(D,M+(26-m) \sum_{i=1}^{3}P_i \right).$$
Also, $C(D,G) \subset C(D,G)^{\bot}$ if $s \leq 13$.
Now we apply Theorem \ref{th:stab}.
\begin{theorem}\label{Th:GGK21Quantum_3}
With the same notation as above, consider the $2^6$-ary code $C(D,s\sum_{i=1}^{3}P_i )$ from the curve $GGK2(2,6)$.
Assume that $6\leq m \leq 13$.
Then there exists a quantum code with parameters
$$[[\,216,\;k_m,\;3m-18\,]]_{2^{6}},\qquad k_m=\begin{cases} 196, \ if \ m=6, \\ 192-6(m-7), \ if \ 7\leq m\leq 13. \end{cases}$$
\end{theorem}

\item $n=5$. Here,
\begin{align*}
f(X)={}&X^{660} + X^{627} + X^{594} + X^{528} + X^{495} + X^{396} + X^{363} + X^{330} + X^{132} + X^{66} + 1,
\end{align*}
and 
$$
f^\prime (X)=X^{626} + X^{494} + X^{362}.
$$
Since $P_1,P_2,P_3$ are not zeros of $f^\prime$ we have $(f^\prime(z))_\infty=626(q^2-q)\sum_{i=1}^{3} P_i$ and hence
$$\deg(f^{\prime}(z))_0 =626(q+1)(q^2-q)=3756.$$
Now, the function 
$$ \sum_{\xi \in \Xi}\frac{1}{z-\xi}=\frac{f^{\prime}(z)}{f(z)}.$$
has principal divisor 
$$M-D+68 \sum_{i=1}^{3}P_i,$$
where $M$ is the zero divisor of $f^{\prime}(z)$  and 
$$D=(f(z))_0= \sum_{P \in \mathbb{P}_{2^{10}}(GGK2(2,5))\setminus \mathbb{P}_{2^{2}}(GGK2(2,5)) }P$$ has degree $3960$.
The principal divisor of 
$$\omega=\sum_{\xi \in \mathcal{A}}\frac{1}{z-\xi}dz$$
is 
$$M-D+98 \sum_{i=1}^{3}P_i.$$
Consider the multi-point divisor $G=s \sum_{i=1}^{3} P_i$.
By Theorem \ref{Main:Swiss} and its proof,
$$C(D,G)^{\bot}=C\left(D,M+(98-m) \sum_{i=1}^{3}P_i \right).$$
Also, $C(D,G) \subset C(D,G)^{\bot}$ if $m \leq 49$.
Now we apply Theorem \ref{th:stab}.
\begin{theorem}\label{Th:GGK21Quantum_5}
With the same notation as above, consider the $2^{10}$-ary code $C(D,m\sum_{i=1}^{3}P_i )$ from the curve $GGK2(2,10)$.
Assume that $30\leq m \leq 49$.
Then there exists a quantum code with parameters
$$[[3960,k_m, 3m-90]]_{2^{6}},\qquad k_m=\begin{cases} 3868, \ if \ m=30, \\ 3864-6(m-31), \ if \ 31\leq m\leq 49. \end{cases} $$
\end{theorem}
\end{itemize}

\section{Comparisons}\label{sec:comp}

\begin{corollary}
The $[[N,k,d]]_{q^6}$-codes constructed in Theorem \ref{Th:GKQuantum} are pure.  If in addition $s\geq 7q^5-14q^3+7q^2+12$, then they do not satisfy Condition \eqref{Dis:GV}.
\end{corollary}
\begin{proof}
Firstly, note that all the $[[N,k,d]]_{q^6}$-codes of Theorem \ref{Th:GKQuantum} satisfy $N\equiv k \pmod 2$. 

Also, the codes are pure. In fact, $C(D,G)$ is an $[N_1,k_1,d_1]_q$ code, with $N_1=q^8-q^6+q^5$, $k_1=s-\frac{(q^3+1)(q^2-2)}{2}$, $d_1\geq q^8-q^6+q^5-s$. It is readily seen that $d_1>k_1+1$ and by Corollary \ref{Corollary} the quantum codes are pure.

As $N-k+2=2s -q^5+2q^3-q^2+4$, the left-hand side of Condition \eqref{Dis:GV} reads
$$\frac{q^{6(N-k+2)}-1}{q^{12}-1}<\frac{q^{6(N-k+2)}}{q^{12}-1}=\frac{q^{6(2s -q^5+2q^3-q^2+4)}}{q^{12}-1},$$
whereas the right-hand side is larger than 
\begin{eqnarray*}
\binom{N}{d-1}(q^{12}-1)^{d-2}&=&\binom{N}{d-1}\frac{(q^{12}-1)^{d-1}}{q^{12}-1} >\left(\frac{N}{d-1}\right)^{d-1}\frac{(q^{12}-1)^{d-1}}{q^{12}-1} \\
&>&\left(\frac{N}{d-1}\right)^{d-2}\frac{q^{12(d-1)}}{q^{12}-1}\geq \frac{q^{12(s-q^5+2q^3-q^2+1)}}{q^{12}-1}q^{d-2},
\end{eqnarray*}
where we used that $N/(d-1)\geq q$, which is implied by $s\leq q^7+q^4-2q^3+q^2-1$ and hence by the hypothesis $s\leq \frac{q^7-q^6+q^5+q^4-2q^3+q^2-2}{2}$.

From $s\geq 7q^5-14q^3+7q^2+12$ follows $d-2+12(s-q^5+2q^3-q^2+1)\geq 6(2s -q^5+2q^3-q^2+4)$; therefore, the left-hand side is smaller than the right-hand side and Condition \eqref{Dis:GV} is not satisfied.
\end{proof}

\section{Acknowledgments*}
The research of D. Bartoli and G. Zini was partially supported  by the Italian National Group for Algebraic and Geometric Structures and their Applications (GNSAGA - INdAM).

\end{document}